\documentclass[12pt]{article}

\usepackage{bm}
\usepackage{mathrsfs}
\usepackage{amsmath}
\usepackage{amsfonts}
\usepackage{mathtools}
\usepackage{amsthm}

\usepackage[margin=1in]{geometry}

\usepackage[english]{babel}
\usepackage{hyperref}
\usepackage{cleveref}
\usepackage{graphicx}
\usepackage{subcaption}
\usepackage{wrapfig}
\usepackage{float}
\usepackage{color}
\usepackage{verbatim}
\usepackage{algorithm}
\usepackage{algorithmic}
\usepackage{multirow}
\usepackage[titletoc, title]{appendix}

\usepackage{wasysym}

\hypersetup{colorlinks=true}

\usepackage{amsopn}

\newcommand{\del}{\partial}

\newtheorem{lemma}{Lemma}

\newtheorem{proposition}{Proposition}

\newtheorem{assumption}{Assumption}

\newtheorem{remark}{Remark}

\crefname{theorem}{Theorem}{Theorems}
\crefname{lemma}{Lemma}{Lemmas}
\crefname{corollary}{Corollary}{Corollaries}
\crefname{assumption}{Assumption}{Assumptions}

\crefname{algorithm}{Algorithm}{Algorithm}
\crefname{section}{Section}{Sections}
\crefname{table}{Table}{Tables}
\crefname{figure}{Figure}{Figures}
\crefname{equation}{}{}
\crefrangeformat{equation}{\textup{(#3#1#4--#5#2#6)}}

\numberwithin{equation}{section}

\graphicspath{{images/}}
\makeatletter
\renewcommand*{\@fnsymbol}[1]{\ifcase#1\or*\else\@arabic{\numexpr#1-1\relax}\fi}
\makeatother

\title{Neural network optimal feedback control with enhanced closed loop stability\thanks{This work was supported with funding from the Air Force Office of Scientific Research (AFOSR) under grant FA9550-21-1-0113.}}

\author{Tenavi Nakamura-Zimmerer\footnote{Department of Applied Mathematics, Baskin School of Engineering, University of California, Santa Cruz (tenakamu@ucsc.edu).}
\and
Qi Gong\thanks{Professor, Department of Applied Mathematics, Baskin School of Engineering, University of California, Santa Cruz.}
\and
Wei Kang\thanks{Professor, Department of Applied Mathematics, Naval Postgraduate School, Monterey, CA.}}

\allowdisplaybreaks

\begin{document}

\maketitle

\begin{abstract}
Recent research has shown that supervised learning can be an effective tool for designing optimal feedback controllers for high-dimensional nonlinear dynamic systems. But the behavior of these neural network (NN) controllers is still not well understood. In this paper we use numerical simulations to demonstrate that typical test accuracy metrics do not effectively capture the ability of an NN controller to stabilize a system. In particular, some NNs with high test accuracy can fail to stabilize the dynamics. To address this we propose two NN architectures which locally approximate a linear quadratic regulator (LQR). Numerical simulations confirm our intuition that the proposed architectures reliably produce stabilizing feedback controllers without sacrificing optimality. In addition, we introduce a preliminary theoretical result describing some stability properties of such NN-controlled systems.
\end{abstract}

\section{Introduction}
\label{sec: intro}

\begin{figure}[t!]
\centering
\includegraphics[width = 0.65 \columnwidth]{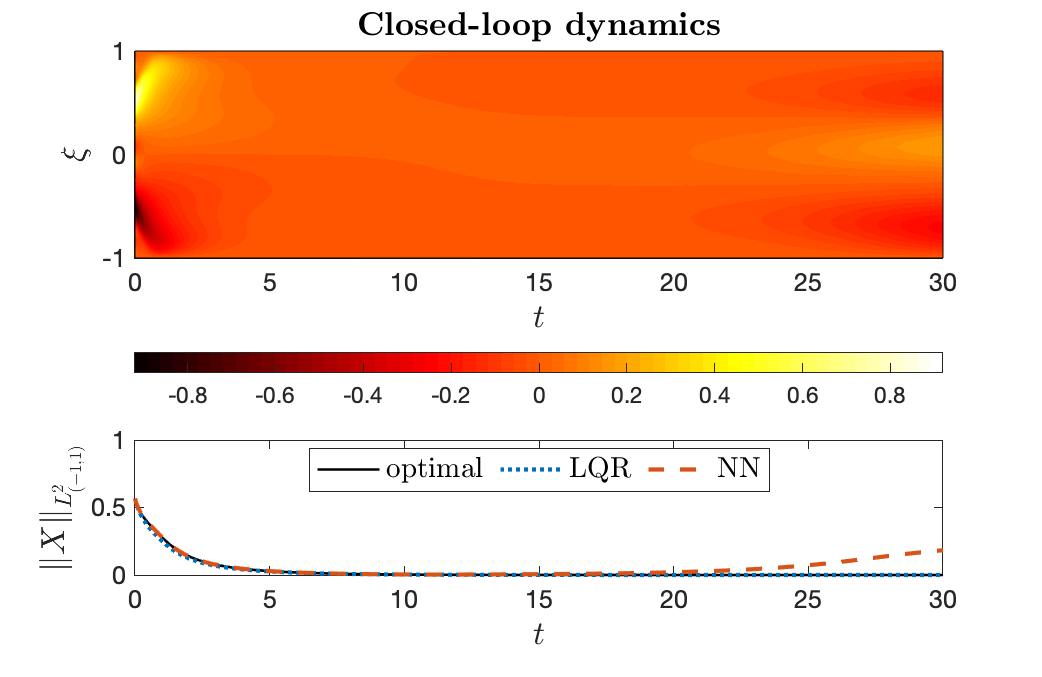}
\includegraphics[width = 0.65 \columnwidth]{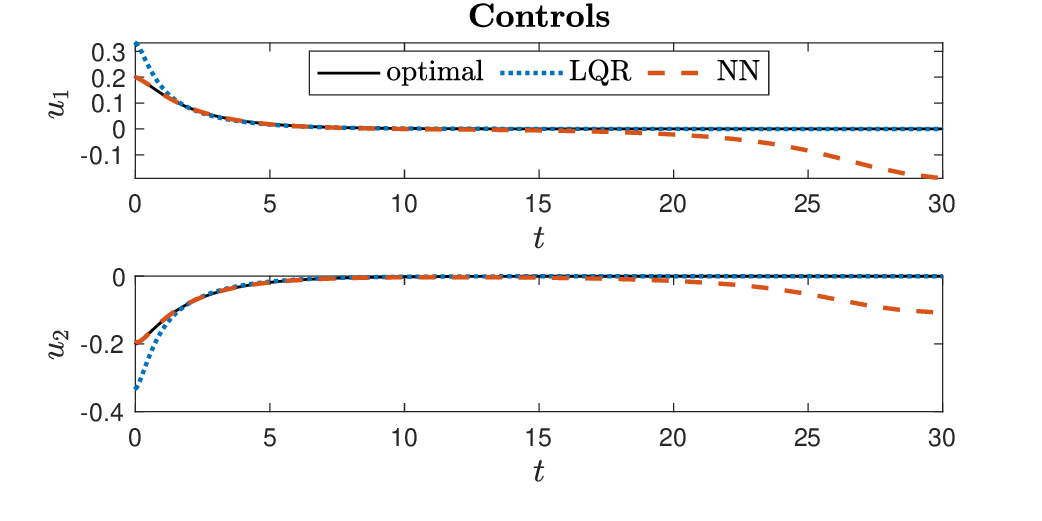}
\caption{Simulation of the Burgers'-type PDE \cref{eq: PS Burgers OCP} with a standard NN feedback controller (no LQR component) exhibiting instability. The top plot shows the evolution of the state $X (t, \xi)$, where $\xi$ is the spatial variable.}
\label{fig: burgers not ultimately bounded}
\end{figure}

Designing optimal feedback controllers for high-dimensional nonlinear systems remains an outstanding challenge for the control community. Even when the system dynamics are known, to design such controllers one needs to solve a Hamilton-Jacobi-Bellman (HJB) partial differential equation (PDE), whose dimension is the same as that of the state space. This leads to the well-known curse of dimensionality, which rules out traditional discretization-based approaches. In recent years, research on supervised learning has demonstrated the promise of these methods for handling challenging, high-dimensional problems.

The core idea of supervised learning is to generate data by solving many open loop optimal control problems and then fit a model to this data set, thus obtaining an approximate optimal feedback controller. Various specific model design and training approaches have been developed within this framework. Earlier work \cite{Kang2015,Kang2017_COA} uses sparse grid interpolation to approximate the solution of the HJB equation -- called the {\em value function} -- and its gradient, which is used to compute the optimal feedback control. This line of work has been futher developed using nonlinear regression with NNs \cite{Izzo2019, Izzo2021_lowthrust, Nakamura2020, Nakamura2021_SIAM, Nakamura2021_LCSS} and sparse polynomials \cite{Azmi2021}, significantly increasing the maximum feasible problem dimension. A variation of the method in \cite{Nakamura2021_SIAM} is proposed by \cite{Chen2020}, in which the value gradient is directly approximated without learning the value function itself, while \cite{Sanchez2018, Tailor2019, Li2020, Izzo2019, Izzo2021_lowthrust} use NNs to directly approximate the control policy without solving the HJB equation.

There are also several closely-related research directions which do not quite fall into the same framework. For example, \cite{Albi2021} propose learning solutions to state-dependent Riccati equations, which yields a suboptimal feedback control law. Many NN-based methods attempt to solve the HJB PDE in the least-squares sense by minimizing the residual of the PDE and boundary conditions at randomly sampled collocation points \cite{Khalaf2005, Tassa2007, Sirignano2018}. Lastly, \cite{Han2018_PNAS, Onken2021, Kunisch2021} propose self-supervised learning approaches to solve the HJB equation along its characteristics without generating any data. Of course, many other methods have been proposed for solving HJB equations and designing feedback controllers, but in the present work we focus on supervised learning approaches as these explicitly quantify error with respect to the optimal control.

Despite promising developments in the methodology, much less work has been done to study and improve the stability and reliability of these NN controllers. To see why this is needed, if we train a set of NNs to control a Burgers'-type PDE system \cref{eq: PS Burgers OCP}, a surprisingly large fraction of these fail to stabilize the system despite having good test accuracy. \Cref{fig: burgers not ultimately bounded} shows a closed loop simulation with one such controller where the NN-controlled trajectory closely tracks the optimal (stable) trajectory until $t \approx 15$ when it goes unstable. Undesirable behavior like this obviates the need for better understanding, more rigorous testing, and more reliable algorithms.

Previously, \cite{Tailor2019} have also pointed out that test accuracy incompletely characterizes the performance of NN controllers, and suggest some more practical evaluations of optimality and stability. Ref. \cite{Izzo2021_stability} study linear stability near a desired equilibrium, linear time delay stability, and stability around a nominal trajectory using high order Taylor maps. Finally, \cite{Nakamura2021_LCSS} propose {\em QRnet}, an NN architecture incorporating an LQR which makes training more reliable.

The purpose of this paper is twofold: first, to bring attention to stability issues with NN-controlled systems; and second, to propose some NN architectures which can potentially mitigate these challenges. We start by describing the problem setting in \cref{sec: problem setting}. In \cref{sec: stability by construction} we propose $\lambda${\em -QRnet} and $u${\em -QRnet}, two NN architectures which retain the stability and robustness properties of {\em QRnet} while ensuring that the desired system equilibrium is always achieved. We emphasize that these controllers are designed not just for stability, but also optimality. In \cref{sec: numerical results} we apply several practical closed loop stability and optimality tests to demonstrate the advantages of the proposed NN architectures. As a testbed we use the Burgers'-type PDE system \cref{eq: PS Burgers OCP}, which is nonlinear, open loop unstable, and high-dimensional. This leads us to \cref{sec: main probability result} in which we consider a new theoretical perspective which probabilistically relates test accuracy and system stability, qualitatively describing what we see in practice that NNs with similar test error can sometimes be stable or unstable. A summary and directions for future work are given in \cref{sec: conclusion}.


\section{Problem setting}
\label{sec: problem setting}

We focus our attention on infinite-horizon nonlinear optimal control problems (OCPs) of the form
\begin{equation}
\label{eq: OCP}
\left \{
\begin{array}{cl}
\underset{\bm u (\cdot)}{\text{minimize}} & J \left[ \bm u (\cdot) \right] = \displaystyle \int_0^\infty \mathcal L (\bm x , \bm u) dt , \\
\text{subject to} & \dot {\bm x} (t) = \bm f (\bm x, \bm u) , \\
	& \bm x (0) = \bm x_0 , \\
	& \bm u (t) \in \mathbb U .
\end{array}
\right .
\end{equation}
Here $\bm x : [0, \infty) \to \mathbb R^n$ is the state, $\bm u : [0, \infty) \to \mathbb U \subseteq \mathbb R^m$ is the control, and $\bm f : \mathbb R^n \times \mathbb U \to \mathbb R^n$ is a vector field which is continuously differentiable ($\mathcal C^1$) in $\bm x$ and $\bm u$. We consider box control constraints which can be expressed as
\begin{equation}
\mathbb U = \left \{ \bm u \in \mathbb R^m \left| u_{\text{min}, i} \leq u_i \leq u_{\text{max}, i} , i = 1, \dots, m \right. \right \} ,
\end{equation}
for vectors $\bm u_{\text{max}}, \bm u_{\text{min}} \in \mathbb R^m$ containing the minimum and maximum values for $\bm u (\cdot)$; and we consider running costs $\mathcal L : \mathbb R^n \times \mathbb U \to [0,\infty)$ of the form
\begin{equation}
\label{eq: quadratic cost}
\mathcal L (\bm x, \bm u) = q \left( \bm x - \bm x_f \right) + \left( \bm u - \bm u_f \right)^T \bm R \left( \bm u - \bm u_f \right) .
\end{equation}
Here $\left( \bm x_f , \bm u_f \right) \in \mathbb R^n \times \mathbb U$ is a (possibly unstable) equilibrium of the dynamics such that $\bm f \left( \bm x_f, \bm u_f \right) = \bm 0$, $\bm R \in \mathbb R^{m \times m}$ is a positive definite matrix, and $q : \mathbb R^n \to [0, \infty)$ is a smooth, positive semi-definite function with the Taylor series approximation
\begin{equation}
q \left( \bm x - \bm x_f \right) \approx \left( \bm x - \bm x_f \right)^T \bm Q \left( \bm x - \bm x_f \right) ,
\quad
\bm Q \coloneqq \left. \frac{\del^2 q}{\del \bm x^2} \right|_{\bm x_f} .
\end{equation}
This standard cost function is a natural choice for regularization or set-point tracking problems where we want to stabilize the objective state $\bm x_f$. We make the standard assumption that $\bm u_f$ is contained in an open subset of $\mathbb U$. We also assume the dynamics $\bm f (\cdot)$ are known and that the OCP \cref{eq: OCP} is well-posed, i.e. that an optimal control $\bm u^* (t)$ exists such that $J \left[ \bm u^* (\cdot) \right] < \infty$ and $\lim_{t\rightarrow \infty} \mathcal L \left( \bm x^*(t), \bm u^*(t) \right) = 0$.

\subsection{The Hamilton-Jacobi-Bellman equation}

Due to real-time application requirements, we would like to design a closed loop feedback controller, $\bm u^* (t) = \bm u^* \left( \bm x (t) \right)$, which can be evaluated online given any measurement of $\bm x$. The mathematical framework for designing such an optimal feedback policy is the HJB equation.

Let the value function $V : \mathbb R^n \to \mathbb R$ be the optimal cost-to-go of \cref{eq: OCP} starting at the point $\bm x (0) = \bm x$:
\begin{equation}
\label{eq: value function}
V (\bm x) \coloneqq J \left[ \bm u^* (\cdot) \right] .
\end{equation}
The value function \cref{eq: value function} is the unique viscosity solution \cite{Crandall1983} of the steady state HJB PDE,
\begin{equation}
\label{eq: HJB}
\left \{ \begin{array}{l}
\displaystyle \min_{\bm u \in \mathbb U} \left\{ \mathcal L(\bm x, \bm u) + \frac{\del V}{\del \bm x} \bm f (\bm x, \bm u) \right\} = 0 , \\
V \left( \bm x_f \right) = 0 .
\end{array} \right .
\end{equation}
If \cref{eq: HJB} can be solved (in the viscosity sense), then it provides both necessary and sufficient conditions for optimality. Next we define the Hamiltonian
\begin{equation}
\label{eq: Hamiltonian}
\mathcal H (\bm x, \bm , \bm u) \coloneqq \mathcal L (\bm x, \bm u) + \bm \lambda^T \bm f (\bm x, \bm u) ,
\end{equation}
for $\bm \lambda \in \mathbb R^n$. Given the gradient of value function, $V_{\bm x} \coloneqq \left[ \del V / \del \bm x \right]^T$, we know that the optimal control must satisfy the Hamiltonian minimization condition
\begin{equation}
\label{eq: Hamiltonian minimization condition}
\bm u^* (\bm x) = \bm u^* \left( \bm x; V_{\bm x} (\bm x) \right)
	= \underset{\bm u \in \mathbb U}{\text{arg min}} \, \mathcal H \left( \bm x, V_{\bm x}, \bm u \right) .
\end{equation}
Thus if we can solve \cref{eq: HJB}, the optimal feedback control is obtained online as the solution of \cref{eq: Hamiltonian minimization condition}.

Because solving \cref{eq: HJB} for general nonlinear systems is extremely challenging, a common approach is to linearize the dynamics about $\left( \bm x = \bm x_f, \bm u = \bm u_f \right)$ to get the approximation
\begin{equation}
\label{eq: linearized dynamics}
\left \{ \begin{array}{l}
\dot{\bm x} \approx \bm A \left( \bm x - \bm x_f \right) + \bm B \left( \bm u - \bm u_f \right) , \\
\bm A \coloneqq \left. \frac{\del \bm f}{\del \bm x} \right|_{\bm x_f , \bm u_f} ,
\quad
\bm B \coloneqq \left. \frac{\del \bm f}{\del \bm u} \right|_{\bm x_f , \bm u_f} .
\end{array} \right .
\end{equation}
Under the standard conditions that $(\bm A, \bm B)$ is controllable and $\left( \bm A, \bm Q^{1/2} \right)$ is observable, then the value function of the OCP with linear dynamics \cref{eq: linearized dynamics} and quadratic cost \cref{eq: quadratic cost} is
\begin{equation}
\label{eq: LQR value}
V^{\text{LQR}} (\bm x) = (\bm x - \bm x_f)^T \bm P (\bm x - \bm x_f) ,
\end{equation}
where $\bm P \in \mathbb R^{n \times n}$ is a positive definite matrix satisfying the Riccati equation,
\begin{equation}
\label{eq: CARE}
\bm Q + \bm A^T  \bm P  + \bm {PA} - \bm P \bm{BR}^{-1} \bm B^T \bm P = \bm 0 .
\end{equation}
Furthermore, the resulting state feedback controller is linear with constant gain:
\begin{equation}
\label{eq: LQR control}
\bm u^{\text{LQR}} (\bm x) = \bm u_f - \bm K (\bm x - \bm x_f) ,
\quad
\bm K = \bm R^{-1} \bm B^T \bm P .
\end{equation}

Sufficiently near the equilibrium $\bm x_f$, the LQR value function $V^{\text{LQR}} (\cdot)$ and linear controller $\bm u^{\text{LQR}} (\cdot)$ are good approximations of the true value function $V(\cdot)$ and optimal control $\bm u^* (\cdot)$. But further away from $\bm x_f$, the control is suboptimal and in some cases may even fail to stabilize the nonlinear dynamics. For this reason we are interested in computing the full {\em nonlinear} optimal feedback control $\bm u^* (\cdot)$ over a semi-global domain.

\subsection{Pontryagin's Minimum Principle}

To circumvent the challenge of directly solving the HJB equation \cref{eq: HJB}, we can leverage the necessary conditions for optimality well-known in optimal control as Pontryagin's Minimum Principle (PMP). This is a two-point BVP, which for the OCP \cref{eq: OCP} takes the form \cite{Pontryagin1987}
\begin{equation}
\label{eq: BVP}
\lim_{t_f \to \infty}
\left \{ \begin{array}{ll}
\dot{\bm x} (t) = \bm f (\bm x, \bm u^* (\bm x; \bm \lambda)) ,
	& \bm x (0) = \bm x_0 , \\
\dot{\bm \lambda} (t) = - \mathcal H_{\bm x} (\bm x, \bm \lambda, \bm u^* (\bm x; \bm \lambda)) ,
	& \bm \lambda (t_f) = \bm 0 .
\end{array} \right.
\end{equation}
Here $\bm \lambda : [0, \infty) \to \mathbb R^n$ is called the {\em costate}. If we assume that the solution of \cref{eq: BVP} is optimal, then along the trajectory $\bm x = \bm x^* (t; \bm x_0)$ we have
\begin{equation}
\label{eq: control and value from BVP}
\left \{ \begin{array}{ll}
V (\bm x) = \int_t^\infty \mathcal L \left( \bm x (s) , \bm u^* (s) \right) ds , \\
V_{\bm x} (\bm x) = \bm \lambda (t) ,
\quad \bm u^* (\bm x) = \bm u^* (t) .
\end{array} \right.
\end{equation}
In the finite-horizon case, the finite-horizon BVP describes the characteristics of the time-dependent HJB equation. While the stationary HJB equation \cref{eq: HJB} does not have characteristics in the same sense, by viewing \cref{eq: HJB} as the infinite-horizon limit of the usual time-dependent HJB equation, we can see that it maintains the same relationship with the infinite-horizon BVP \cref{eq: BVP}.

In general, the BVP \cref{eq: BVP} admits multiple solutions. So while the characteristics of the value function satisfy \cref{eq: BVP}, there may be other solutions to these equations which are sub-optimal and thus not characteristics. In certain problems the characteristics can also intersect, giving rise to non-smooth value functions and difficulties in applying \cref{eq: control and value from BVP}. Optimality of solutions to \cref{eq: BVP} can be guaranteed under some convexity conditions (see e.g. \cite{Mangasarian1966}). For most dynamical systems it is difficult to verify such conditions globally, but we can guarantee optimality locally around an equilibrium point \cite{Lukes1969}. Addressing the challenge of global optimality is beyond the scope of the present work, so in this paper we assume that solutions of \cref{eq: BVP} are optimal. Under this assumption, the relationship between PMP and the value function given in \cref{eq: control and value from BVP} holds everywhere.

Note the supervised learning approaches based on PMP can still be applied even when optimality cannot be verified. In such cases PMP remains the prevailing tool for finding {\em candidate optimal} solutions, and from these the proposed method yields a stabilizing feedback controller which satisfies necessary conditions for local optimality.


\section{Optimal feedback design with $\lambda${\em -QRnet} and $u${\em -QRnet} architectures}
\label{sec: stability by construction}

In this paper we consider feedback controllers designed by means of {\em supervised learning} to approximate the optimal control. That is, we learn a feedback control policy
$$
\widehat{\bm u} : \mathbb R^n \to \mathbb U ,
\qquad
\widehat{\bm u} (\bm x) \approx \bm u^* (\bm x) ,
$$
based on data generated with PMP. As discussed in \cref{sec: intro}, one can directly model the optimal control policy, or alternatively, model the value function or its gradient and substitute the approximate gradient into \cref{eq: Hamiltonian minimization condition} to obtain the control.

While previous work has clearly demonstrated the potential of supervised learning with NNs as a means of overcoming the curse of dimensionality in optimal control, NNs are notoriously ``black boxes'' and their behavior -- especially when implemented in the closed loop system -- is hard to predict. Even if we can train a highly accurate NN, it can still fail to stabilize the closed loop system. Further, under an NN-based feedback control the goal state $\bm x_f$ may not be an equilibrium. This motivates us to design feedback controllers with {\em built-in} stability properties.

In this paper we introduce two NN architectures, $\lambda${\em -QRnet} and $u${\em -QRnet}, which automatically make $\bm x_f$ an equilibrium and locally approximate the LQR control \cref{eq: LQR control}. \Cref{sec: new QRnet architecture} describes the proposed architectures, both of which combine a linear term from LQR with an NN. The LQR terms are good approximations of the optimal control near $\bm x_f$, and intuitively enhance local stability. Meanwhile, the NNs are intended to capture nonlinearities and thereby learn the nonlinear optimal feedback over a large domain.

Once we have chosen a model structure, supervised learning can be broken down into three steps: data generation (\cref{sec: data generation}), training (\cref{sec: model training}), and finally model evaluation against test data (\cref{sec: test accuracy}). In \cref{sec: numerical results} we illustrate a more rigorous test regimen specifically for control design, by which we demonstrate that the proposed controllers yield locally stabilizing controllers which closely approximate the nonlinear optimal feedback law.

\subsection{Model architectures}
\label{sec: new QRnet architecture}

We start with $\lambda${\em -QRnet}, which can be readily used if the optimal control is available analytically as a function of $V_{\bm x} (\cdot)$. In this case we implement the feedback control as
\begin{equation}
\label{eq: grad-QRnet control}
\widehat{\bm u} (\bm x) = \bm u^* \left( \bm x; \widehat{\bm \lambda} (\bm x) \right) ,
\end{equation}
where $\widehat{\bm \lambda} (\bm x) \approx V_{\bm x} (\bm x)$ is an approximation of value gradient of the form
\begin{equation}
\label{eq: grad-QRnet architecture}
\widehat{\bm \lambda} (\bm x) = 2 \bm P \left( \bm x - \bm x_f \right) + \bm{\mathcal N} \left( \bm x ; \bm \theta \right) - \bm{\mathcal N} \left( \bm x_f ; \bm \theta \right) .
\end{equation}
Here $\bm{\mathcal N} : \mathbb R^n \times \mathbb R^p \to \mathbb R^n$ is a nonlinear vector-valued $\mathcal C^1$ function (such as an NN) parameterized by $\bm \theta \in \mathbb R^p$. Notice the linear component $2 \bm P \left( \bm x - \bm x_f \right)$ which is the LQR value gradient.

Next we introduce $u${\em -QRnet}, which can be used even if we cannot easily write down $\bm u^* (\cdot)$ in terms of $V_{\bm x} (\cdot)$. In this case we directly approximate the optimal control by

\begin{equation}
\label{eq: ctrl-QRnet architecture}
\widehat{\bm u} (\bm x) = \sigma \left( \bm u_f - \bm K \left( \bm x - \bm x_f \right) + \bm{\mathcal N} \left( \bm x ; \bm \theta \right) - \bm{\mathcal N} \left( \bm x_f ; \bm \theta \right) \right) ,
\end{equation}
where now $\bm{\mathcal N} : \mathbb R^n \times \mathbb R^p \to \mathbb R^m$ instead of $\mathbb R^n$, and $\sigma : \mathbb R^m \to \mathbb R^m$ is a generalized logistic function which smoothly saturates the control:
\begin{equation}
\label{eq: saturation function}
\sigma (\bm u) \coloneqq \bm u_{\text{min}} + \frac{\bm u_{\text{max}} - \bm u_{\text{min}}}{1 + c_1 \exp \left[ - c_2 \left( \bm u - \bm u_f \right) \right]} .
\end{equation}
Here $\bm u_{\text{max}}, \bm u_{\text{min}} \in \mathbb R^m$ are vectors containing the (componentwise) minimum and maximum values for $\bm u (\cdot)$, and $c_1, c_2 > 0$ are constants which we set as
\begin{equation}
\label{eq: c1 c2}
c_1 = \frac{\bm u_{\text{max}} - \bm u_f}{\bm u_f - \bm u_{\text{min}}} ,
\quad
c_2 = \frac{\bm u_{\text{max}} - \bm u_{\text{min}}}{\left( \bm u_{\text{max}} - \bm u_f \right) \left( \bm u_f - \bm u_{\text{min}} \right)} .
\end{equation}
It is straightforward to verify that these choices of $c_1, c_2$ satisfy $\sigma \left( \bm u_f \right) = \bm u_f$ and $\frac{\del \sigma}{\del \bm u} \left( \bm u_f \right) = 1$. Consequently, $\sigma (\cdot)$ smoothly imposes saturation constraints while preserving the unsaturated control behavior near $\bm u_f$. Notice again the linear component $\bm K \left( \bm x - \bm x_f \right)$ which comes from the LQR control.

One of the main advantages of these architectures is that they automatically make the goal state $\bm x_f$ an equilibrium. This is achieved by the $\left[ - \bm{\mathcal N} \left( \bm x_f ; \bm \theta \right) \right]$ term in \cref{eq: grad-QRnet architecture,eq: ctrl-QRnet architecture}, which is also suggested in \cite{Kunisch2021}. This property is formalized in the following proposition, whose proof is straightforward.

\begin{proposition}[$\lambda${\em -QRnet} and $u${\em -QRnet} equilibria]
\label{prop: new QRnet equilibrium}
Assume that $\bm u_f$ is in an open ball contained in $\mathbb U$ and $\widehat{\bm u} (\cdot)$ is a feedback policy specified by \crefrange{eq: grad-QRnet control}{eq: grad-QRnet architecture} or \crefrange{eq: ctrl-QRnet architecture}{eq: c1 c2}. Then $\bm x_f$ is an equilibrium of the NN-controlled system, $\dot{\bm x} = \bm f \left( \bm x, \widehat{\bm u} (\bm x) \right)$.
\end{proposition}

\begin{proof}
Evaluating \cref{eq: grad-QRnet architecture} at $\bm x = \bm x_f$ gives $\widehat{\bm \lambda} \left( \bm x_f \right) = \bm 0 = V_{\bm x} \left( \bm x_f \right)$ which implies
$$
\widehat{\bm u} \left( \bm x_f \right)
	= \bm u^* \left( \bm x ; \widehat{\bm \lambda} \left( \bm x_f \right) \right)
	= \bm u^* \left( \bm x ; V_{\bm x} \left( \bm x_f \right) \right)
	= \bm u_f .
$$
Alternatively, evaluating \crefrange{eq: ctrl-QRnet architecture}{eq: c1 c2} at $\bm x = \bm x_f$ yields $\widehat{\bm u} \left( \bm x_f \right) = \sigma \left( \bm u_f \right) = \bm u_f$. It follows that $\dot{\bm x}
	= \bm f \left( \bm x_f , \widehat{\bm u} \left( \bm x_f \right) \right)
	= \bm f \left( \bm x_f, \bm u_f \right)
	= \bm 0$ for both architectures \crefrange{eq: grad-QRnet control}{eq: grad-QRnet architecture} and \crefrange{eq: ctrl-QRnet architecture}{eq: c1 c2}.
\end{proof}

\begin{remark}
Implementing the proposed controllers requires solving the Riccati equation \cref{eq: CARE} to compute the $\bm P$ and $\bm K$ matrices, which can be done numerically for well-posed OCPs. We note that while linearizations of \cref{eq: grad-QRnet architecture,eq: ctrl-QRnet architecture} at $\bm x_f$ do not exactly recover the LQR gain, in practice the difference is usually small. Thus local stability is usually preserved by LQR's large gain and phase margins (see \cref{sec: local stability}).
\end{remark}

\subsection{Related architectures}
\label{sec: QRnet}

The proposed model architectures are similar to {\em QRnet} \cite{Nakamura2021_LCSS}, which approximates the control based on the gradient of a value function model
\begin{equation}
\label{eq: QRnet}
\widehat V (\bm x) = \frac{1}{\gamma} \log \left[ 1 + \gamma V^{\text{LQR}} \left( \bm x \right) \right] + \mathcal N \left( \bm x ; \bm \theta \right) ,
\end{equation}
where $\gamma > 0$ is a trainable parameter and $\mathcal N : \mathbb R^n \times \mathbb R^p \to \mathbb R$ is an NN. Like the controllers introduced in this paper, we can see that \cref{eq: QRnet} combines an NN with the LQR approximation $V^{\text{LQR}} (\cdot)$ to improve performance near $\bm x_f$. Although {\em QRnet} empirically improves stability properties, it does not guarantee that the goal state $\bm x_f$ will be stable, let alone an equilibrium. For this reason, in preliminary experiments we implemented a similar value function model
\begin{equation}
\label{eq: sQRnet}
\widehat V (\bm x) = \frac{1}{\gamma} \log \left[ 1 + \gamma V^{\text{LQR}} \left( \bm x \right) \right] \left[ 1 + \mathcal N \left( \bm x ; \bm \theta \right) \right] .
\end{equation}
Straightforward computations show that an NN controller $\widehat{\bm u} (\bm x) = \bm u^* \left( \bm x ; \widehat V_{\bm x} (\bm x) \right)$ using the gradient $\widehat V_{\bm x} (\cdot)$ of \cref{eq: sQRnet} automatically makes $\bm x_f$ an equilibrium and locally approximates the LQR value function and control up to a scalar multiple $\kappa = 1 + \mathcal N \left( \bm x_f ; \bm \theta \right)$. One can apply the gain margin properties of LQR \cite[Theorem~4]{Wong1977} to show that this controller locally stabilizes $\bm x_f$ as long as $\kappa > 1/2$, which always occurs in practice as a result of NN training. Despite these favorable local properties we found experimentally that, when evaluated on a large test domain, the closed loop performance was no better than a plain NN and significantly worse than than {\em QRnet}.

In contrast, the proposed $\lambda${\em -QRnet} and $u${\em -QRnet} architectures ensure that the goal state $\bm x_f$ is an equilibrium while retaining the enhanced local stability and semi-global performance properties of {\em QRnet}. In addition, while $\lambda${\em -QRnet} and $u${\em -QRnet} approximate higher-dimensional output functions which require larger output layers in the NN, since we do not need to take the gradients of $\widehat V (\cdot)$ with respect to $\bm x$ to evaluate $\widehat{\bm u} (\cdot)$, training time is actually reduced. (see \cref{fig: burgers training}).

\subsection{Data generation}
\label{sec: data generation}

To generate training and testing data for supervised learning, we solve the open loop OCP \cref{eq: OCP} for a set of (randomly sampled) initial conditions. For each open loop solution we apply \cref{eq: control and value from BVP} at each point along the controlled trajectory $\bm x = \bm x (t ; \bm x_0)$ to obtain input-output pairs $\bm x^{(i)}$, $\left( V_{\bm x} \left( \bm x^{(i)} \right), \bm u^* \left( \bm x^{(i)} \right) \right)$, and where the superscript $(i)$ is the sample index. Aggregating data from all open loop solutions\footnote{Note that there is no need to distinguish data from different trajectories as the value function and optimal feedback control are time-independent.}, we obtain a data set
\begin{equation}
\label{eq: data set}
\mathcal D_{\text{train}} = \left \{ \bm x^{(i)}, V_{\bm x} \left( \bm x^{(i)} \right), \bm u^* \left( \bm x^{(i)} \right) \right \}_{i=1}^{N_{\text{train}}} .
\end{equation}
We recall that each open loop OCP can be solved independently without knowledge of nearby solutions, and is related to the closed loop solution by PMP. Methods based on this idea are referred to as {\em causality-free} \cite{Kang2015}.

Algorithms for solving the open loop OCP \cref{eq: OCP} can be broadly classified as {\em indirect} or {\em direct} methods \cite{Betts1998}. Indirect methods take the ``optimize then discretize'' approach, solving the OCP \cref{eq: OCP} by way of the two-point BVP \cref{eq: BVP}. As such, these methods provide both costate and control data and can thus be used either for learning the optimal control law directly or learning the value function.

Direct methods, on the other hand, take the ``discretize then optimize'' approach, transforming the OCP \cref{eq: OCP} into a large nonlinear programming problem. One significant advantage of direct methods is their ability to easily handle more complicated OCPs, such as those with path constraints. In the context of supervised learning, \cite{Tailor2019, Li2020} use Hermite-Simpson direct collocation to generate data for finite-horizon OCPs. Alternatively, Radau pseudospectral methods \cite{Ross2006, Fahroo2008} are ideal for solving infinite-horizon open loop OCPs, though they have not yet been used for supervised learning. We note that pseudospectral methods have the added benefit of the covector mapping principle \cite{Ross2005, Fahroo2008}, which allows one to extract the costates from the solution of the discretized OCP and thus generate data for learning the value function or its gradient.

In this paper we put aside the details of how best to generate data, and assume that we can generate accurate data of the form \cref{eq: data set} as we desire. For more detailed discussions on solving infinite-horizon open loop OCPs \cref{eq: OCP} and data generation methods in supervised learning, we refer the reader to \cite{Kang2021, Nakamura2021_LCSS, Betts1998, Fahroo2008} and references therein.

\subsection{Model training}
\label{sec: model training}

Once a set of training data is available, the next step is training -- i.e. data-driven optimization. If we denote the model parameters by $\bm \theta$ (i.e. the weights and biases of the NN), then the NN is trained by minimizing a mean squared error loss function:
\begin{equation}
\label{eq: NN optimization}
\bm \theta = \underset{\bm \theta}{\text{arg min}} \frac{1}{N_{\text{train}}} \sum_{i=1}^{N_{\text{train}}} \left \Vert \widehat{\bm u} \left( \bm x^{(i)}; \bm \theta \right) - \bm u^* \left( \bm x^{(i)} \right) \right \Vert_2^2 .
\end{equation}
As is standard in machine learning, the models learn on data which has been scaled to the range $[-1,1]$, and the output is accordingly rescaled to the original domain when ultimately used for control.

When training $\lambda${\em -QRnet}, one could optionally augment the loss function \cref{eq: NN optimization} with an additional term to learn the value gradient \cite{Izzo2019, Izzo2021_lowthrust, Nakamura2020, Nakamura2021_SIAM, Chen2020}, for example
\begin{equation}
\text{loss}_{\bm \lambda} (\bm \theta) = \frac{1}{N_{\text{train}}} \sum_{i=1}^{N_{\text{train}}} \left \Vert \widehat{\bm \lambda} \left( \bm x^{(i)}; \bm \theta \right) - V_{\bm x} \left( \bm x^{(i)} \right) \right \Vert_2^2 ,
\end{equation}
and/or a term to minimize the residual of the HJB equation \cref{eq: HJB}. Both $\lambda${\em -QRnet} and $u${\em -QRnet} approaches would also work well in conjunction with active learning methods \cite{Nakamura2021_SIAM}. For both $\lambda${\em -QRnet} and $u${\em -QRnet} we carry out numerical optimization using L-BFGS \cite{Byrd1995} as we find that it is faster than stochastic gradient descent for moderately-sized data sets and NNs.

\subsection{Quantifying model accuracy}
\label{sec: test accuracy}

To quantify the accuracy of the model, we generate a second test data set, $\mathcal D_{\text{test}}$, from {\em independently drawn} initial conditions. During training, the NN sees only data points from the training set $\mathcal D_{\text{train}}$, while $\mathcal D_{\text{test}}$ is reserved for evaluating approximation accuracy after training. A typical metric used in supervised learning is the mean $\ell^2$ error,
\begin{equation}
\label{eq: ML2}
\text{mean } \ell^2 \coloneqq
	\frac{1}{N_{\text{test}}} \sum_{i=1}^{N_{\text{test}}} \left \Vert \widehat{\bm u} \left( \bm x^{(i)} \right) - \bm u^* \left( \bm x^{(i)} \right) \right \Vert_2 ,
\end{equation}
where $N_{\text{test}}$ denotes the number of test points $\bm x^{(i)} \in \mathcal D_{\text{test}}$. A low test error indicates that the NN generalizes well, i.e. it did not overfit the training data. Although less commonly reported, the maximum $\ell^2$ error,
\begin{equation}
\label{eq: max L2}
\text{max } \ell^2 \coloneqq
	\max_{i \in \{ 1, \dots, N_{\text{test}} \}} \left \Vert \widehat{\bm u} \left( \bm x^{(i)} \right) - \bm u^* \left( \bm x^{(i)} \right) \right \Vert_2 ,
\end{equation}
can be more relevant and convenient in the context of determining system stability (see \cref{sec: main probability result}). Throughout the paper we often refer to \cref{eq: max L2} as $\delta_N$, where $N = N_{\text{test}}$.

But even with a low maximum test error, there is a chance that the NN could still perform poorly when implemented in the closed loop system. For this reason we believe that test metrics like \cref{eq: ML2,eq: max L2} are insufficient in the context of control design; we should instead focus on rigorous closed loop stability and performance tests such as those presented in \cref{sec: numerical results}.


\section{Numerical results}
\label{sec: numerical results}

In this section we empirically compare the closed loop stability and optimality properties of LQR, a standard NN which models $\widehat V (\bm x) \approx V (\bm x)$, a standard NN denoted as $u$-NN which directly learns $\widehat{\bm u} (\bm x) \approx \bm u (\bm x)$, {\em QRnet} \cref{eq: QRnet}, $\lambda${\em -QRnet} \crefrange{eq: grad-QRnet control}{eq: grad-QRnet architecture}, and $u${\em -QRnet} \crefrange{eq: ctrl-QRnet architecture}{eq: c1 c2}. We present results for three different tests:
\begin{enumerate}
\item linear stability analysis near $\bm x_f$ (\cref{sec: local stability});
\item Monte Carlo (MC) nonlinear stability (\cref{sec: numerical semi-global stability});
\item MC optimality analysis (\cref{sec: numerical suboptimality}).
\end{enumerate}
Such tests are of course familiar to the control community, but we believe it is worth emphasizing their importance for control design since more rigorous and realistic testing is needed in order to start trusting NN controllers in real-world applications. We also note that these tests are just a starting point: further examples include stabilization time \cite{Tailor2019}, time delay stability \cite{Izzo2021_stability}, and robustness to measurement noise, disturbances, and parameter variations.

The numerical results clearly illustrate that standard NNs are {\em not} consistently stable, even when they have good approximation accuracy. Meanwhile, the results suggest that the proposed $\lambda${\em -QRnet} and $u${\em -QRnet} architectures improve closed loop system stability, essentially decoupling this from the model's approximation accuracy. The $u$-NN controllers are just as likely to be unstable as the NNs which model the value function, indicating that the benefits from the $\lambda${\em -QRnet} and $u${\em -QRnet} architectures come from the added LQR structure -- not because they model the value gradient or control instead the value function.

\subsection{Control of unstable Burgers'-type PDE}
\label{sec: burgers description}

As a testbed we revisit the modified Burgers' stabilization OCP from \cite{Nakamura2021_LCSS}. This is a high-dimensional problem formulated by pseudospectral discretization of an unstable version of a Burgers' PDE. Similar benchmark problems have recently been considered in \cite{Kalise2018, Borggaard2020, Nakamura2020, Nakamura2021_SIAM}.

Briefly, the Burgers' stabilization OCP considered in \cite{Nakamura2021_LCSS} can be summarized as
\begin{equation}
\label{eq: PS Burgers OCP}
\left \{ \begin{array}{cl}
\underset{\bm u (\cdot)}{\text{min.}} & J \left[ \bm u (\cdot) \right] = \displaystyle\int_0^\infty \left( \bm x^T \bm {Qx} + \bm u^T \bm {Ru} \right) dt  , \\
\text{s.t.} & \dot{\bm x} = - \displaystyle \frac{1}{2} \bm D \bm x \circ \bm x + \nu \bm D^2 \bm x + \bm \alpha \circ \bm x \circ e^{- \beta \bm x} + \bm {B u}.
\end{array} \right.
\end{equation}
Here $\bm x : [0, \infty) \to \mathbb R^n$ represents the PDE state $X (t, \xi)$ collocated at spatial coordinates $\xi_j = \cos \left( j \pi / n \right)$, $j = 1, \dots, n$, $\bm u : [0, \infty) \to \mathbb R^m$ is the control, $\bm D \in \mathbb R^{n \times n}$ is the Chebyshev differentiation matrix, $\bm Q \in \mathbb R^{n \times n}, \bm R \in \mathbb R^{m \times m}$ are diagonal positive definite matrices, and ``$\circ$'' denotes elementwise multiplication. The parameters $\nu, \beta > 0, \bm \alpha \in \mathbb R^n$, and $\bm B \in \mathbb R^{n \times m}$ are defined in \cite{Nakamura2021_LCSS}, and we take $n = 64$ and $m = 2$.
 
Using the {\em LQR warm start} data generation strategy \cite{Nakamura2021_LCSS} we generate training and test data sets by solving the BVP \cref{eq: BVP} for randomly generated initial conditions. To get models with varying approximation accuracies, we generate training data sets with different numbers of trajectories (8, 16, 32, 64, and 128). Naturally, the larger the training data set the better the NN model becomes. The standard value function NNs and {\em QRnet} are trained as described in \cite{Nakamura2021_LCSS}. To be consistent, all NNs use five hidden layers with 32 neurons each.

Note that because data generation depends on random sampling and \cref{eq: lambda optimization} is a highly non-convex optimization problem, results can vary considerably for different random seeds. To account for this, for each different data set size we conduct ten trials with different randomly generated training trajectories and NN weight initializations. We evaluate test metrics \cref{eq: ML2,eq: max L2} on an independent test data set containing 400 trajectories totaling $N_{\text{test}} = 66060$ data points.

Optimization of \cref{eq: NN optimization} is carried out with L-BFGS \cite{Byrd1995}, which stops when the relative change in the loss is sufficiently small. \Cref{fig: burgers training} shows training times of the NNs, where we see that models which directly approximate the control are significantly faster to train. This may be surprising at first since these NNs approximate higher-dimensional functions, and thus have larger output layers. But since we do not have to take gradients of $\widehat V (\cdot)$ with respect to $\bm x$ to evaluate $\widehat{\bm u} (\cdot)$, training time is actually reduced.

\begin{figure}[t!]
\centering
\includegraphics[width = 0.65 \columnwidth]{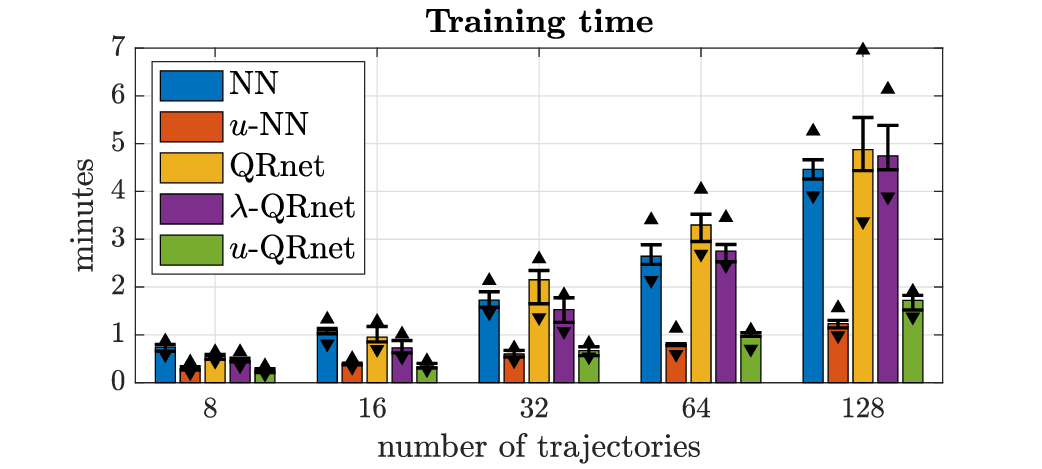}
\caption{Training time for different NN architectures depending on the amount of training data. Bar heights show the medians over ten trials, error bars indicate the 25th and 75th percentiles, and triangles indicate minimum and maximum values.}
\label{fig: burgers training}
\end{figure}

\subsection{Local stability analysis}
\label{sec: local stability}

Let $\widehat{\bm u} (\cdot)$ be a given feedback control policy and $\bar{\bm x}$ be an equilibrium of the controlled system. Note that in general $\bar{\bm x} \neq \bm x_f$. Assuming that $\bm f (\cdot)$ and $\widehat{\bm u} (\cdot)$ are differentiable, the closed loop dynamics can be approximated near $\bar{\bm x}$ by
\begin{equation}
\label{eq: general linearized closed loop model}
\small
\dot{\bm x} \approx \mathcal A \left( \bm x - \bar{\bm x} \right) ,
\ 
\mathcal A
	\coloneqq \left. \frac{\del \bm f}{\del \bm x} \right|_{\bar{\bm x}, \widehat{\bm u} \left( \bar{\bm x} \right)} + \left. \frac{\del \bm f}{\del \bm u} \right|_{\bar{\bm x}, \widehat{\bm u} \left( \bar{\bm x} \right)} \left. \frac{\del \widehat{\bm u}}{\del \bm x} \right|_{\bar{\bm x}} .
\end{equation}
Thus, after synthesizing a feedback controller $\widehat{\bm u} (\cdot)$, we can use a root-finding algorithm to solve $\bm 0 = \bm f \left( \bm x, \widehat{\bm u} (\bm x) \right)$ for an equilibrium $\bar{\bm x}$. Then we can check for local stability by seeing if the closed loop Jacobian $\mathcal A$ is Hurwitz. As noted in \cite{Izzo2021_stability}, one benefit of using an NN controller with differentiable activation functions is that the closed loop dynamics are $\mathcal C^1$. This provides an exact closed loop Jacobian which makes it easier to solve for the equilibrium $\bar{\bm x}$ and allows us to use tools from linear systems theory to characterize local stability.

\Cref{fig: burgers eigenvalues} shows the real part of the most positive eigenvalue of the closed loop Jacobians for the set of NNs trained as described in \cref{sec: burgers description} above. We observe that standard NNs must be trained to a high level of test accuracy before they are even locally stable. Achieving this bare minimum local stability requirement necessitates a large training data set and a longer training time. On the other hand, {\em QRnet}, $\lambda${\em -QRnet}, and $u${\em -QRnet} all yield local stability even when they are trained on small data sets. This gives us confidence in the design process: when training such models we can focus on optimal control performance while worrying less about local stability. After training, we can easily check for local stability as described above.

\begin{figure}[t!]
\centering
\includegraphics[width = 0.65 \columnwidth]{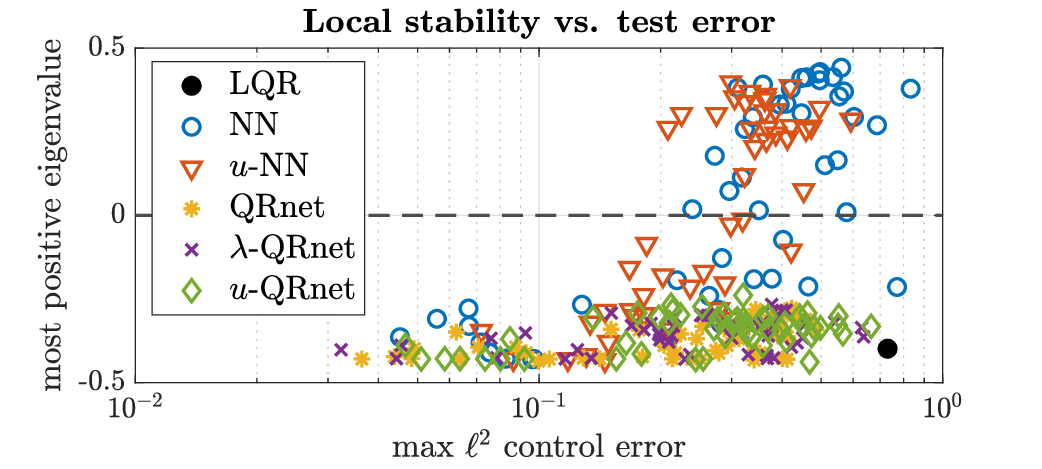}
\caption{Real part of most positive closed loop Jacobian eigenvalue at an equilibrium near $\bm x_f$. Each marker represents a single model. All models below the dashed line are locally stable.}
\label{fig: burgers eigenvalues}
\end{figure}

\subsection{MC nonlinear stability analysis}
\label{sec: numerical semi-global stability}

This test and the one in \cref{sec: numerical suboptimality} are based on MC closed loop simulations of $N_{\text{MC}} = 100$ trajectories. Initial conditions $\bm x_0^{(i)}$, $i = 1, \dots, N_{\text{MC}}$ are randomly selected with norm $\left \Vert \bm x_0^{(i)} \right \Vert = 1.2 \approx \max_{\bm x^{(j)} \in \mathcal D_{\text{train}}} \left \Vert \bm x^{(j)} \right \Vert$, placing them at the edge of the training domain where the NNs may be less accurate and the trajectories harder to control. The same set of $N_{\text{MC}} = 100$ initial conditions was used for every controller.

For each controller and initial condition we simulate the closed loop system until it reaches a steady state or exceeds a large final time. We call the largest final state norm, $\max_i \left \Vert \bm x \left( t_f ; \bm x_0^{(i)} \right) \right \Vert$, the {\em worst-case failure}. If this is sufficiently small (i.e. all trajectories converge to $\bm x_f$) then the system is likely semi-globally stable or ultimately bounded. Conversely, if the controller fails to stabilize even one trajectory then we cannot rely on it.

\Cref{fig: burgers final state} shows the worst-case failures for the set of NNs trained as described in \cref{sec: burgers description} above. We again observe that standard NNs must be trained to a high level of test accuracy before they successfully stabilize the system. More interestingly, although stability on average improves with test accuracy, this is not always the case: some very accurate standard NNs fail to consistently stabilize the system. We further discuss this phenomenon in \cref{sec: main probability result}.

In contrast to the standard NNs, all but two {\em QRnet} controllers stabilize the system, and all the $\lambda${\em -QRnet} and $u${\em -QRnet} controllers successfully stabilize the system, even those with very low test accuracy. Furthermore, the steady states which these systems reach are all zero (up to integration tolerances), while the original {\em QRnet} steady states are larger because the {\em QRnet}-controlled dynamics will have non-zero equilibria. Similarly to \cref{sec: local stability}, these empirical results suggest that the proposed architectures make the control design more reliable, consistently yielding a stabilizing control law even with small data sets and short training times.

\begin{figure}[t!]
\centering
\includegraphics[width = 0.65 \columnwidth]{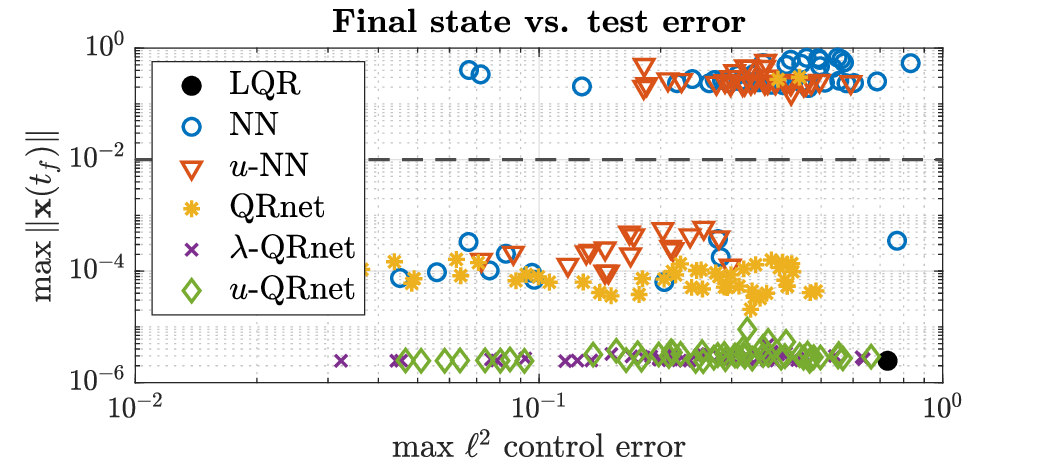}
\caption{Worst-case norm of final state over $N_{\text{MC}} = 100$ simulations. Each marker represents a single model and the dashed line separates unstable (top) from stable (bottom) systems.}
\label{fig: burgers final state}
\end{figure}

\subsection{MC optimality analysis}
\label{sec: numerical suboptimality}

In this work we are interested in both stability and optimality. Optimality is quantified by the accumulated cost for each controller and across all MC simulations. As a reference we can compare these costs to the optimal costs, $V \left( \bm x_0^{(i)} \right)$, computed by solving the BVP \cref{eq: BVP}. \Cref{fig: burgers suboptimality} shows the results of this analysis for the same set of MC simulations conducted in \cref{sec: numerical semi-global stability}. Note that we only show models which were deemed stable in \cref{sec: numerical semi-global stability}, as unstable models can accumulate infinite cost. We can see a clear correlation between higher test accuracy and better performance. All the NN controllers (if stable) follow this trend, and moreover, usually perform better than a simple LQR controller. It follows that {\em QRnet}, $\lambda${\em -QRnet}, and $u${\em -QRnet} improve stability without sacrificing optimality.

\begin{figure}[t!]
\centering
\includegraphics[width = 0.65 \columnwidth]{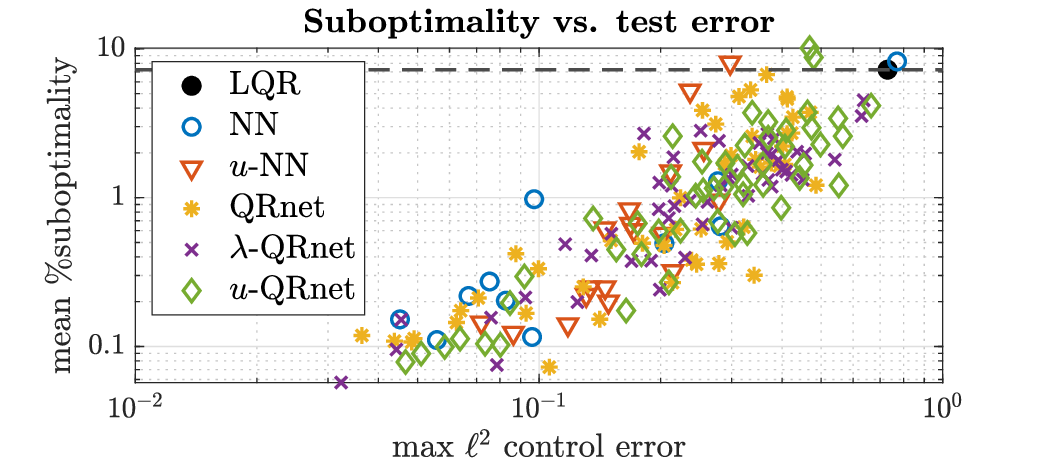}
\includegraphics[width = 0.65 \columnwidth]{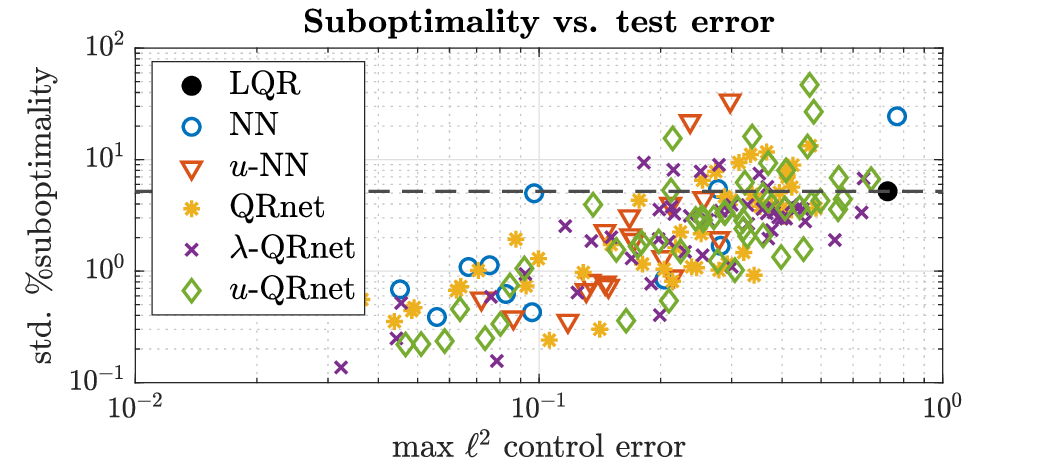}
\caption{Mean percent extra cost vs. BVP data over $N_{\text{MC}} = 100$ simulations. Each marker represents a single model and the dashed line shows the performance of LQR.}
\label{fig: burgers suboptimality}
\end{figure}


\section{Probabilistic stability analysis}
\label{sec: main probability result}

In \cref{sec: numerical semi-global stability} we noticed that plain NNs with high test accuracy could sometimes be unstable. This motivates us to try to futher understand some stability properties of deterministic closed loop systems with control policies $\widehat{\bm u} (\cdot)$ that are trained and tested on randomly-generated data. It is important to emphasize that $\widehat{\bm u} (\cdot)$, once trained, is completely deterministic, but its final performance is influenced by and measured by random sampling. This gives rise to an interesting situation where although the closed loop dynamics are deterministic, we can only speak about stability in a probabilistic sense. Our goal is to provide intuitive yet novel qualitative insights into the behavior of these systems.

\subsection{Deterministic stability}
\label{sec: deterministic stability}

We start with a standard stability result for perturbations to exponentially stable systems. Without loss of generality let the objective equilibrium state be $(\bm x_f = \bm 0, \bm u_f = \bm 0)$. This simplifies notation and easily extends to arbitrary equilibrium states. Next let $r > 0$ and denote the open ball
\begin{equation}
\mathbb B_r \coloneqq \left \{ \left. \bm x \in \mathbb R^n \right| \Vert \bm x \Vert_2 < r \right \} \subseteq \mathbb R^n .
\end{equation}
We will make the following assumptions about the optimally-controlled system, $\dot{\bm x} = \bm f^*_{\text{cl}} (\bm x) \coloneqq \bm f \left( \bm x, \bm u^* (\bm x) \right)$:
\begin{assumption}
\label{assume: Lipschitz Jacobian}
The Jacobian matrix $\del \bm f^*_{\text{cl}} / \del \bm x$ is Lipschitz throughout the closure $\overline{\mathbb B}_r$.
\end{assumption}
\begin{assumption}
\label{assume: exponentially stable}
The goal state $\bm x_f$ is an exponentially stable equilibrium of the optimally-controlled system.
\end{assumption}

We remark that these assumptions are {\em not} restrictive in the present context. In particular, they hold under the standard assumptions that the dynamics are locally $\mathcal C^1$ in $\bm x$ and $\bm u$, the value function is locally $\mathcal C^2$, and that the Ricatti equation \cref{eq: CARE} admits a stabilizing solution. In addition, when we design the cost function \cref{eq: quadratic cost} we generally choose this to satisfy \cref{assume: exponentially stable} as exponential stability provides important robustness properties. 

The following lemma characterizes the stability of the closed loop dynamics if we replace the optimal feedback control $\bm u^* (\cdot)$ with some $\widehat{\bm u} (\cdot)$. Since there is no guarantee that the origin remains an equilibrium under the new control law, stability is stated in terms of ultimate boundededness, the property that trajectories starting close to the origin stay close to the origin.

\begin{lemma}[Ultimate boundedness of perturbed systems]
\label{lemma: ultimate boundedness}
Suppose that \cref{assume: Lipschitz Jacobian,assume: exponentially stable} are satisfied. Let $\widehat{\bm u} : \mathbb R^n \to \mathbb U$ be a continuous feedback control policy and assume that its true maximum control approximation error,
\begin{equation}
\label{eq: true control error}
\delta \coloneqq \max_{\bm x \in \overline{\mathbb B}_r} \left \Vert \widehat{\bm u} (\bm x) - \bm u^* (\bm x) \right \Vert_2 ,
\end{equation}
is sufficiently small. Then trajectories of the closed loop system starting sufficiently close to the origin are ultimately bounded with ultimate bound proportional to $\delta$.
\end{lemma}

\begin{proof}
\Cref{assume: Lipschitz Jacobian,assume: exponentially stable} satisfy the conditions of \cite[Theorem~4.14]{Khalil2002} which establishes the existence of a $\mathcal C^1$ Lyapunov function $W : \mathbb B_r \to [0, \infty)$ for the optimally-controlled system $\dot{\bm x} = \bm f \left( \bm x, \bm u^* (\bm x) \right)$. This $W (\cdot)$ is of quadratic type, meaning that there exist constants $k_1, k_2, k_3, k_4 > 0$ such that\footnote{If $\bm Q$ is positive definite then we can take the quadratic Lyapunov function for the optimal system to be $W (\bm x) = V^{\text{LQR}} (\bm x)$, and the constants $k_1, \dots, k_4$ can be computed based on the eigenvalues of $\bm P$ and $\bm Q$.}
\begin{equation}
\label{eq: converse lyapunov properties}
\renewcommand{\arraystretch}{1.5}
\left \{ \begin{array}{l}
k_1 \Vert \bm x \Vert_2^2 \leq W (\bm x) \leq k_2 \Vert \bm x \Vert_2^2 ,\\
\left[ \frac{\del W}{\del \bm x} (\bm x) \right] \bm f \left( \bm x, \bm u^* (\bm x) \right) \leq - k_3 \Vert \bm x \Vert_2^2 , \\
\left \Vert \frac{\del W}{\del \bm x} (\bm x) \right \Vert_2 \leq k_4 \Vert \bm x \Vert_2 .
\end{array} \right.
\end{equation}

Now we rewrite the closed loop dynamics as a perturbation of the optimally controlled system, $\dot{\bm x} = \bm f \left( \bm x, \bm u^* (\bm x) \right)$:
\begin{equation}
\label{eq: NN-controlled system perturbation}
\dot{\bm x} = \bm f \left( \bm x, \widehat{\bm u} (\bm x) \right)
	= \bm f \left( \bm x, \bm u^* (\bm x) \right) + \left[ \bm f \left( \bm x, \widehat{\bm u} (\bm x) \right) - \bm f \left( \bm x, \bm u^* (\bm x) \right) \right] .
\end{equation}
Then for all $\bm x \in \overline{\mathbb B}_r$ the perturbation satisfies
\begin{equation}
\left \Vert \bm f \left( \bm x, \widehat{\bm u} (\bm x) \right) - \bm f \left( \bm x, \bm u^* (\bm x) \right) \right \Vert_2
\leq L_{\bm u} \delta ,
\end{equation}
where $L_{\bm u}$ is the Lipschitz constant of $\bm f (\cdot)$ with respect to $\bm u$. Hence if there exists $\theta \in (0,1)$ such that
\begin{equation}
\label{eq: ultimate bound error condition}
\delta < \delta^+
	\coloneqq \frac{\theta k_3}{L_{\bm u} k_4} \sqrt{\frac{k_1}{k_2}} r ,
\end{equation}
then \cite[Lemma~9.2]{Khalil2002} guarantees that for all trajectories of the perturbed system \cref{eq: NN-controlled system perturbation} with $\left \Vert \bm x_0 \right \Vert_2 < r \sqrt{k_1 / k_2}$, there exists some finite time $T = T \left( \bm x_0 \right)$ such that
\begin{equation}
\begin{dcases*}
\left \Vert \bm x (t; \bm x_0) \right \Vert_2 \leq K \Vert \bm x_0 \Vert_2 \exp \left( - \alpha t \right) ,
& $0 \leq t < T$, \\
\left \Vert \bm x (t; \bm x_0) \right \Vert_2 \leq B ,
& $T \leq t$,
\end{dcases*}
\end{equation}
with
\begin{equation}
K = \sqrt{\frac{k_2}{k_1}} ,
\quad
\alpha = \frac{(1 - \theta) k_3}{2 k_2} ,
\quad
B = \frac{L_{\bm u} k_4}{\theta k_3} \sqrt{\frac{k_2}{k_1}} \delta .
\end{equation}
\end{proof}

\subsection{Maximum error estimation}
\label{sec: global optimization}

Since we do not have access to the true optimal control $\bm u^* (\cdot)$, we cannot compute the true maximum error $\delta$ needed to apply \cref{lemma: ultimate boundedness}. Hence in practice we need to estimate this using test data. Let $\delta_N$ be the maximum error \cref{eq: max L2} for a set of test points $\bm x^{(i)} \in \overline {\mathbb B}_r$, $i = 1, \dots, N$. Since we only have $\delta_N$ and not $\delta$, this leads us to two important questions:
\begin{enumerate}
\item How can we accurately approximate $\delta_N \approx \delta$ using a reasonably-sized test data set?
\item With what level of confidence can we rely on $\delta_N$ to characterize the stability of the closed loop system?
\end{enumerate}
The simplest approach to estimating $\delta_N$ is with independent uniform samples. But as we will see in \cref{lemma: maximum estimate,eq: uniform F}, this is not sample-efficient in high dimensions. Since generating each test point requires solving the BVP \cref{eq: BVP}, it can become prohibitively costly to generate sufficient data for testing in this way. Global optimization may offer a more tractable approach, and many methods have been developed for this purpose. Since there are too many algorithms to review here, we refer the reader to e.g. \cite{Locatelli2013} for a summary, and note that any method chosen for this application should work without gradient information and not require too many function evaluations. Clearly this is a difficult problem, though one upside is that each BVP which we solve can generate an entire trajectory of test points which can subsequently be used for training new models.

To work towards an answer to the second question, we start with the following lemma which considers the problem of estimating the maximum value of a continuous function on a compact domain using independent samples. It is easy to see how this specializes to computing maximum test errors, and in \cref{sec: stability and test accuracy} we apply this result to stability analysis. Estimates of the gap between $\delta_N$ and $\delta$ for correlated test points\footnote{Such correlations can be induced by using entire trajectories and by various global optimization algorithms.} are beyond the scope of this work.

\begin{lemma}[Maximum estimation]
\label{lemma: maximum estimate}
Let $\mathbb D \subset \mathbb R^n$ be a compact set with non-zero volume and let $g : \mathbb D \to \mathbb [0, \infty)$ be a continuous, non-negative function with maximum value $\delta \coloneqq \max_{\bm x \in \mathbb D} g (\bm x)$. Suppose that $\bm x^{(i)}, i = 1, \dots, N$, are independently sampled according to a probability distribution $\mu (\bm x)$ supported everywhere on $\mathbb D$, and let $\delta_N \coloneqq \max_{i \in \{ 1, \dots, N \}} g \left( \bm x^{(i)} \right)$. Then for any $\epsilon > 0$, we have
\begin{equation}
\Pr \left( \delta > \delta_N + \epsilon \right) = \left[ \Pr \left( \delta - g (\bm x) > \epsilon \right) \right]^N
\end{equation}
with $\Pr \left( \delta - g (\bm x) > \epsilon \right) < 1$, and thus $\{ \delta_N \}_{N=1}^\infty \to \delta$ in probability.
\end{lemma}

\begin{proof}
For any $N, \epsilon > 0$ we compute
\begin{align*}
\Pr \left( \delta > \delta_N + \epsilon \right)
	=& \Pr \left( \delta - \max_{i \in \{ 1, \dots, N \}} g \left( \bm x^{(i)} \right) > \epsilon \right) \\
	=& \Pr \left( \delta - g \left( \bm x^{(i)} \right) > \epsilon, i = 1, \dots, N \right) \\
	=& \left[ \Pr \left( \delta - g (\bm x) > \epsilon \right) \right]^N ,
\end{align*}
where we have used the fact that $\bm x^{(i)}$ are independent. Let
\begin{equation}
F (\delta, \epsilon) \coloneqq \Pr \left( \delta - g (\bm x) > \epsilon \right) .
\end{equation}
We claim that $F (\delta, \epsilon) < 1$ for all $\epsilon > 0$. To see this, first since $g (\bm x) \geq 0$ we immediately have $F (\delta, \epsilon) = 0$ for all $\epsilon \geq \delta$. Next if $0 < \epsilon < \delta$ let $\bm z \in \mathbb D$ be any point which achieves the true maximum, i.e. $g (\bm z) = \delta$. By continuity of $g (\cdot)$, we know that there must exist some $d > 0$ such that for all $\bm x$ in the neighborhood $\Vert \bm z - \bm x \Vert_2 < d$, we have
$$
\left| g (\bm z) - g (\bm x) \right| 
	= \left| \delta - g (\bm x) \right|
	= \delta - g (\bm x)
	\leq \epsilon .
$$
We also know that for any $d > 0$, the intersection
$$
\mathbb D \cap \left \{ \bm x \in \mathbb R^n \left| \Vert \bm z - \bm x \Vert_2 < d \right. \right \}
	= \left \{ \bm x \in \mathbb D \left| \Vert \bm z - \bm x \Vert_2 < d \right. \right \}
$$
is non-empty and open. Then because $\mu (\bm x)$ is supported everywhere on $\mathbb D$, it follows that the set $\left \{ \bm x \in \mathbb D \left| \delta - g (\bm x) \leq \epsilon \right. \right \}$ must have non-zero probability mass and hence $F (\delta, \epsilon) < 1$. Therefore
\begin{equation}
\label{eq: converges in probability}
\lim_{N \to \infty} \Pr \left( | \delta - \delta_N | > \epsilon \right) =
	\lim_{N \to \infty} \Pr \left( \delta > \delta_N + \epsilon \right)
	= \lim_{N \to \infty} \left[ F (\delta, \epsilon) \right]^N
	= 0 .
\end{equation}
Furthermore, $\delta \geq \delta_N$ implies $| \delta - \delta_N | = \delta - \delta_N$ and so the sequence $\{ \delta_N \}_{N=1}^\infty$ converges in probability to $\delta$. 
\end{proof}

Eq. \cref{eq: converges in probability} appears promising since $\left[ F (\delta, \epsilon) \right]^N$ decrease exponentially in $N$, but it is worth pointing out that $F (\delta, \epsilon)$ depends on the volume of the sample domain, how smooth $g (\cdot)$ is, and of course the method for sampling $\bm x^{(i)}$. For example, suppose that $\bm x^{(i)}$ are sampled uniformly from $\mathbb D$ so that
\begin{equation}
\label{eq: uniform F}
F (\delta, \epsilon)
	= \frac{\int_{\left \{ \bm x \in \mathbb D \left| \delta - g (\bm x) > \epsilon \right. \right \}} d \bm x}{\int_{\mathbb D} d \bm x} .
\end{equation}
By inspection we can see that $F (\delta, \epsilon)$ is smaller if $g (\bm x)$ is flatter, and conversely $F (\delta, \epsilon) \to 1$ as the domain grows (which happens if we increase the dimension). While $\left[ F (\delta, \epsilon) \right]^N$ does decreases exponentially in $N$ once we fix $\epsilon$ and the test domain $\mathbb D$, because of the strong dependence on the problem dimension in practice we should prefer optimization-based strategies over random sampling.

\subsection{Probabilistic stability based on test accuracy}
\label{sec: stability and test accuracy}

Now we are ready to apply \cref{lemma: maximum estimate} to find the probability that the our error estimate $\delta_N$ is close enough to $\delta$ such that the feedback controller $\widehat{\bm u} (\cdot)$ meets the requirements for \cref{lemma: ultimate boundedness}, and hence sufficient conditions for ultimate boundededness.

\begin{proposition}[Probability of ultimate boundedness]
\label{prop: ultimate bound probability}
Suppose that \cref{assume: Lipschitz Jacobian,assume: exponentially stable} are satisfied and let $\widehat{\bm u} : \mathbb R^n \to \mathbb U$ be a continuous feedback control. Consider test points $\bm x^{(i)}, i = 1, \dots, N$, independently sampled according to some probability distribution supported everywhere on $\overline{\mathbb B}_r$. Let $\delta$ bet the true unknown control error \cref{eq: true control error}, $\delta_N$ be the error estimate \cref{eq: max L2}, and $\delta^+$ be a constant defined in \cref{eq: ultimate bound error condition}. If $\delta_N < \delta^+$, then the probability that the error estimate is accurate enough to determine if the closed loop system satisfies sufficient conditions for ultimate boundedness is given by
\begin{equation}
\label{eq: P (delta N)}
P = 1 - \left[ F \left( \delta, \epsilon_N \right) \right]^N > 0 ,
\end{equation}
where $\epsilon_N \coloneqq \delta^+ - \delta_N$ and
\begin{equation}
\label{eq: general F (eps) control}
F \left( \delta, \epsilon_N \right)
	= \Pr \left( \delta - \left \Vert \widehat{\bm u} (\bm x) - \bm u^* (\bm x) \right \Vert_2 > \epsilon_N \right) 
	< 1 .
\end{equation}
\end{proposition}

\begin{proof}
Noting that \cref{assume: Lipschitz Jacobian} implies $\bm u^* (\cdot)$ is continuous in $\bm x$, \cref{lemma: maximum estimate} yields
$$
P \left( \delta, \delta_N, N \right) \coloneqq \Pr \left( \delta < \delta_N + \epsilon_N \right)
	= 1 - \Pr \left( \delta \geq \delta_N + \epsilon_N \right)
	= 1 - \left[ F \left( \delta, \epsilon_N \right) \right]^N .
$$
Though $\delta$ is fixed, \cref{eq: P (delta N)} tells us the probability that $\delta_N$ is close {\em enough} to $\delta$ so that we can expect \cref{eq: ultimate bound error condition} holds, which by \cref{lemma: ultimate boundedness} implies the trajectories of closed loop system are ultimately bounded.
\end{proof}

While stability is a deterministic property of the system (i.e. the system either is stable or it isn't), we can loosely think of \cref{prop: ultimate bound probability} as a (conservative) probabilistic stability condition based on test error. In general, it may be difficult to apply this result quantitatively because it requires knowledge about the Lyapunov function $W (\cdot)$ and the true maximum error $\delta$. Nevertheless, we believe that \cref{prop: ultimate bound probability} begins to qualitatively explain the phenomenon presented in \cref{sec: numerical results}, where NNs with similar test accuracy can sometimes produce stable systems and other times yield trajectories like the one shown in \cref{fig: burgers not ultimately bounded}.


\section{Summary and future work}
\label{sec: conclusion}

In this paper we have used practical closed loop stability and optimality tests to show that NN feedback controllers can fail to stabilize a system, even when they are trained to a high degree of accuracy. This occurs frequently enough that it cannot be ignored, so to increase the acceptability of NN feedback controllers we need more rigorous testing and more reliable model architectures. {\em QRnet} \cite{Nakamura2021_LCSS} is a first step in this direction, empirically improving stability properties with the addition of an LQR component. $\lambda${\em -QRnet} and $u${\em -QRnet} extend this idea, retaining the enhanced stability and reliablity of {\em QRnet} while reducing computation time and ensuring that the goal state is an equilibrium.

Finally, \cref{sec: main probability result} introduces a new theoretical perspective to qualitatively describe explain the stability properties of NN-controlled dynamic systems. This preliminary result seeks to invoke classical stability results when their conditions cannot be directly checked. In future work we aim to make the theory more practical, as well as develop other perspectives relating density of training trajectories to probabilistic stability metrics. We also intend to explore reasons why and to what extent $\lambda${\em -QRnet} and $u${\em -QRnet} improve system stability. Such theoretical advances will be necessary if supervised learning is to become a reliable and commonly accepted control design method.

\bibliography{bibliography_master}

\begin{thebibliography}{10}

\bibitem{Khalaf2005}
{\sc M.~Abu-Khalaf and F.~L. Lewis}, {\em Nearly optimal control laws for
  nonlinear systems with saturating actuators using a neural network {HJB}
  approach}, Automatica, 41 (2005), pp.~779--791,
  \url{https://doi.org/10.1016/j.automatica.2004.11.034}.

\bibitem{Albi2021}
{\sc G.~Albi, S.~Bicego, and D.~Kalise}, {\em Gradient-augmented supervised
  learning of optimal feedback laws using state-dependent {R}iccati equations},
  2021, \url{https://arxiv.org/abs/2103.04091}.

\bibitem{Azmi2021}
{\sc B.~Azmi, D.~Kalise, and K.~Kunisch}, {\em Optimal feedback law recovery by
  gradient-augmented sparse polynomial regression}, J. Mach. Learn. Res., 22
  (2021), pp.~1--32.

\bibitem{Betts1998}
{\sc J.~T. Betts}, {\em Survey of numerical methods for trajectory
  optimization}, J. Guid., Control, Dyna., 21 (1998), pp.~193--207,
  \url{https://doi.org/10.2514/2.4231}.

\bibitem{Borggaard2020}
{\sc J.~Borggaard and L.~Zietsman}, {\em The quadratic-quadratic regulator
  problem: Approximating feedback controls for quadratic-in-state nonlinear
  systems}, in American Control Conference (ACC), 2020, pp.~818--823,
  \url{https://doi.org/10.23919/ACC45564.2020.9147286}.

\bibitem{Byrd1995}
{\sc R.~H. Byrd, P.~Lu, J.~Nocedal, and C.~Zhu}, {\em A limited memory
  algorithm for bound constrained optimization}, SIAM J. Sci. Comput., 16
  (1995), pp.~1190--1208, \url{https://doi.org/10.1137/0916069}.

\bibitem{Chen2020}
{\sc G.~Chen}, {\em Deep neural network approximations for the stable manifolds
  of the {Hamilton-Jacobi} equations}, 2020,
  \url{https://arxiv.org/abs/2007.15350}.

\bibitem{Crandall1983}
{\sc M.~G. Crandall and P.-L. Lions}, {\em Viscosity solutions of
  {H}amilton-{J}acobi equations}, Trans. Amer. Math. Soc., 277 (1983),
  pp.~1--42, \url{https://doi.org/10.2307/1999343}.

\bibitem{Fahroo2008}
{\sc F.~Fahroo and I.~M. Ross}, {\em Pseudospectral methods for
  infinite-horizon nonlinear optimal control problems}, J. Guid., Control,
  Dyna., 31 (2008), pp.~927--936, \url{https://doi.org/10.2514/1.33117}.

\bibitem{Han2018_PNAS}
{\sc J.~Han, A.~Jentzen, and W.~E}, {\em Solving high-dimensional partial
  differential equations using deep learning}, Proc. Natl. Acad. Sci. USA, 115
  (2018), pp.~8505--8510, \url{https://doi.org/10.1073/pnas.1718942115}.

\bibitem{Izzo2021_lowthrust}
{\sc D.~Izzo and E.~\"{O}zt\"{u}rk}, {\em Real-time guidance for low-thrust
  transfers using deep neural networks}, J. Guid., Control, Dyna.,  (2021),
  pp.~1--13, \url{https://doi.org/10.2514/1.G005254}.

\bibitem{Izzo2019}
{\sc D.~Izzo, E.~\"{O}zt\"{u}rk, and M.~M\"{a}rtens}, {\em Interplanetary
  transfers via deep representations of the optimal policy and/or of the value
  function}, in Genetic and Evolutionary Computation Conference, 2019,
  pp.~1971---1979, \url{https://doi.org/10.1145/3319619.3326834}.

\bibitem{Izzo2021_stability}
{\sc D.~Izzo, D.~Tailor, and T.~Vasileiou}, {\em On the stability analysis of
  deep neural network representations of an optimal state-feedback}, {IEEE}
  Trans. Aerosp. Electron. Syst., 57 (2021), pp.~145--154,
  \url{https://doi.org/10.1109/TAES.2020.3010670}.

\bibitem{Kalise2018}
{\sc D.~Kalise and K.~Kunisch}, {\em Polynomial approximation of
  high-dimensional {Hamilton-Jacobi-Bellman} equations and applications to
  feedback control of semilinear parabolic {PDE}s}, SIAM J. Sci. Comput., 40
  (2018), pp.~A629--A652, \url{https://doi.org/10.1137/17M1116635}.

\bibitem{Kang2021}
{\sc W.~Kang, Q.~Gong, T.~Nakamura-Zimmerer, and F.~Fahroo}, {\em Algorithms of
  data generation for deep learning and feedback design: A survey}, Phys. D,
  (2021), p.~132955, \url{https://doi.org/10.1016/j.physd.2021.132955}.

\bibitem{Kang2015}
{\sc W.~Kang and L.~C. Wilcox}, {\em A causality free computational method for
  {HJB} equations with application to rigid body satellites}, in AIAA Guidance,
  Navigation, and Control Conference, 2015, pp.~1--10,
  \url{https://doi.org/10.2514/6.2015-2009}.

\bibitem{Kang2017_COA}
{\sc W.~Kang and L.~C. Wilcox}, {\em Mitigating the curse of dimensionality:
  Sparse grid characteristics method for optimal feedback control and {HJB}
  equations}, Comput. Optim. Appl., 68 (2017), pp.~289--315,
  \url{https://doi.org/10.1007/s10589-017-9910-0}.

\bibitem{Khalil2002}
{\sc H.~Khalil}, {\em Nonlinear Systems}, Prentice Hall, Upper Saddle River,
  NJ, 3rd~ed., 2002.

\bibitem{Kunisch2021}
{\sc K.~Kunisch and D.~Walter}, {\em Semiglobal optimal feedback stabilization
  of autonomous systems via deep neural network approximation}, {ESAIM} Control
  Optim. Calc. Var., 27 (2021), p.~59,
  \url{https://doi.org/10.1051/cocv/2021009}.

\bibitem{Li2020}
{\sc S.~Li, E.~\"{O}zt\"{u}rk, C.~D. Wagter, G.~C. H.~E. de~Croon, and
  D.~Izzo}, {\em Aggressive online control of a quadrotor via deep network
  representations of optimality principles}, in IEEE International Conference
  on Robotics and Automation (ICRA), 2020, pp.~6282--6287,
  \url{https://doi.org/10.1109/ICRA40945.2020.9197443}.

\bibitem{Locatelli2013}
{\sc M.~Locatelli and F.~Schoen}, {\em Global Optimization}, Society for
  Industrial and Applied Mathematics, Philadelphia, PA, 2013,
  \url{https://doi.org/10.1137/1.9781611972672}.

\bibitem{Lukes1969}
{\sc D.~Lukes}, {\em Optimal regulation of nonlinear dynamical systems}, SIAM
  J. Control, 7 (1969), pp.~75--100, \url{https://doi.org/10.1137/0307007}.

\bibitem{Mangasarian1966}
{\sc O.~L. Mangasarian}, {\em Sufficient conditions for the optimal control of
  nonlinear systems}, SIAM J. Control, 4 (1966), pp.~139--152,
  \url{https://doi.org/10.1137/0304013}.

\bibitem{Nakamura2020}
{\sc T.~Nakamura-Zimmerer, Q.~Gong, and W.~Kang}, {\em A causality-free neural
  network method for high-dimensional {H}amilton-{J}acobi-{B}ellman equations},
  in American Control Conference (ACC), 2020, pp.~787--793,
  \url{https://doi.org/10.23919/ACC45564.2020.9147270}.

\bibitem{Nakamura2021_SIAM}
{\sc T.~Nakamura-Zimmerer, Q.~Gong, and W.~Kang}, {\em Adaptive deep learning
  for high-dimensional {H}amilton--{J}acobi--{B}ellman equations}, SIAM J. Sci.
  Comput., 43 (2021), pp.~A1221--A1247,
  \url{https://doi.org/10.1137/19M1288802}.

\bibitem{Nakamura2021_LCSS}
{\sc T.~Nakamura-Zimmerer, Q.~Gong, and W.~Kang}, {\em {QR}net: Optimal
  regulator design with {LQR}-augmented neural networks}, {IEEE} Control
  Systems Letters, 5 (2021), pp.~1303--1308,
  \url{https://doi.org/10.1109/LCSYS.2020.3034415}.

\bibitem{Onken2021}
{\sc D.~Onken, L.~Nurbekyan, X.~Li, S.~W. Fung, S.~Osher, and L.~Ruthotto},
  {\em A neural network approach for high-dimensional optimal control}, 2021,
  \url{https://arxiv.org/abs/2104.03270}.

\bibitem{Pontryagin1987}
{\sc L.~S. Pontryagin}, {\em Mathematical Theory of Optimal Processes}, vol.~4
  of L.S. Pontryagin selected works, Taylor and Francis, 1987.

\bibitem{Ross2005}
{\sc I.~M. Ross}, {\em A historical introduction to the covector mapping
  principle}, in Adv. Astronautical Sci., vol.~123, 2005, pp.~1257--1278.

\bibitem{Ross2006}
{\sc I.~M. Ross, Q.~Gong, F.~Fahroo, and W.~Kang}, {\em Practical stabilization
  through real-time optimal control}, in American Control Conference (ACC),
  2006, pp.~304--309.

\bibitem{Sanchez2018}
{\sc C.~S\'anchez-S\'anchez and D.~Izzo}, {\em Real-time optimal control via
  deep neural networks: Study on landing problems}, J. Guid., Control, Dyna.,
  41 (2018), pp.~1122--1135, \url{https://doi.org/10.2514/1.G002357}.

\bibitem{Sirignano2018}
{\sc J.~Sirignano and K.~Spiliopoulos}, {\em {DGM}: A deep learning algorithm
  for solving partial differential equations}, J. Comput. Phys., 375 (2018),
  pp.~1339--1364, \url{https://doi.org/10.1016/j.jcp.2018.08.029}.

\bibitem{Tailor2019}
{\sc D.~Tailor and D.~Izzo}, {\em Learning the optimal state-feedback via
  supervised imitation learning}, Astrodynamics, 3 (2019), pp.~361--374,
  \url{https://doi.org/10.1007/s42064-019-0054-0}.

\bibitem{Tassa2007}
{\sc Y.~Tassa and T.~Erez}, {\em Least squares solutions of the {HJB} equation
  with neural network value-function approximators}, {IEEE} Trans. Neural
  Netw., 18 (2007), pp.~1031--1041,
  \url{https://doi.org/10.1109/TNN.2007.899249}.

\bibitem{Wong1977}
{\sc P.~K. Wong and M.~Athans}, {\em Closed-loop structural stability for
  linear-quadratic optimal systems}, {IEEE} Trans. Automat. Control, 22 (1977),
  pp.~94--99, \url{https://doi.org/10.1109/TAC.1977.1101414}.

\end{thebibliography}

\end{document}